\documentclass[a4paper,8pt]{article}

\usepackage{epsfig}
\usepackage{graphicx}
\usepackage{amsfonts}
\usepackage{amssymb}
\usepackage{amsmath}
\usepackage{multicol}
\usepackage[arrow, matrix, curve]{xy}
\usepackage{geometry, color}
\usepackage[cam, a4, center]{crop}
\newtheorem{definition}{Definition}[section]
\newtheorem{satz}{Proposition}[section]
\newtheorem{theorem}{Theorem}[section]
\newtheorem{lemma}{Lemma}[section]
\newtheorem{folgerung}{Corollary}[section]
\newenvironment{proof}{{\bf Proof: } }{ \hfill $\square$ \newline \\ }
\newcommand{\Q}{\mathbb Q}
\newcommand{\C}{\mathbb C}
\renewcommand{\O}{\Omega}
\newcommand{\E}{\mathcal E}

\title{Algebraic cycles on the generic abelian fourfold with polarization of type $(1,2,2,2)$}
\author{R. A. Qui\~nones Estrella \footnote{The autor was supported for this work by DAAD and SFB/TR-45.}}

\begin{document}

\maketitle
\begin{abstract}
In this paper we construct a non-trivial element in the higher Griff{}iths group $Griff ^{3,2}$ for the generic abelian fourfold $A^4$ with polarization of type $(1,2,2,2)$. The key idea is to use that $A^4$ can be realized as a generalized Prym variety and for this reason contains in a natural way some curves i.e. dimension $1$ cycles.
\end{abstract}
\setcounter{section}{-1}
\section{Introduction}

It is a difficult task to construct algebraic cycles which are homologically trivial and not algebraically equivalent to zero, i.e. non-trivial elements in the Griff{}iths groups
\[
 Griff ^r := Ch^r_{hom}/Ch^r_{alg}.
\]
Here $Ch$ means always with rational coefficients, i.e. $Ch\otimes \mathbb Q$
and our varieties are always complex varieties. Torelli \cite{To} has shown that
abelian varieties are quotients of Jacobian varieties. This implies in
particular that abelian varieties contain always
1-dimensional cycles in a natural way.\\
Ceresa \cite{Ce} has shown that for the generic principal polarized complex abelian variety $A^3$ it holds
\[
0\neq Griff ^2(A^3) := Ch^2(A^3)_{hom}/Ch^2(A^3)_{alg}.
\]
The word {\bf generic} here has the following meaning (cf. \cite{BL} pg. 559)

\begin{definition} 
A polarized abelian variety $A$ with polarization of type $(d_1,...,d_g)$ is generic for a property $\mathbb P$ if $[A]\in \mathcal A_g(d_1,..,d_g)$ is outside the union of countably many proper Zariski-closed subsets of $\mathcal A_g(d_1,...,d_g)$ defined by $\mathbb P$, where $\mathcal A_g(d_1,...,d_g)$ is the moduli space of polarized abelian varieties with polarization of type $(d_1,...,d_g)$.
\end{definition}
Ceresa used the fact that $A^3$ is the Jacobian of a curve $C$ and then applied the Abel-Jacobi map to show that the cycle $C-C^-$ is non trivial in $Griff ^2(A^3)$. The explicit result is
\begin{theorem}
Let $g\geq 3$ and $J(C)$ be generic jacobian variety of dimension $g$. Let $W_r$
be the image of the natural map $Sym ^r(C)\longrightarrow J(C)$ and
$W_r^-:=(-1_{J(C)})_*W_r$. If $1\leq r\leq g-2$ then the cycles
$W_r-W_r^-\in Ch^{g-r}(J(C))$ are homologically but not algebraically equivalent
to cero.
\end{theorem}
Let $M$ be the subgroup of $Griff ^2(A^3)$ generated by all the cycles of the form $C-C^-=W_1-W_1^-$. Bardelli \cite{Ba} and Nori \cite{No} have shown
\begin{theorem}
$M$ is not finitely generated.
\end{theorem}
We have also the following result from \cite{Co-Pi}.
\begin{theorem}[Colombo \& Pirola]
$M\neq Griff^2(A^3)$.
\end{theorem}
The results presented so far use that $A^3$ is a jacobian variety and therefore
we can't apply the methods in higher dimension. Nevertheless we can in some
special cases realize abelian varieties as Prym or generalized Prym varieties
which are quotients of Jacobian varieties.\\
This gives a motivation for looking at codimension $n-1$ cycles on $A^n$, with
$A^n$ the generic (no necessary principal) polarized abelian variety of
dimension $n$.\\
Fakhruddin \cite{Fk} used the fact that the principal polarized abelian fivefold $A^5$ is a Prym variety corresponding to a double \'etale cover of a genus $9$ curve over a genus $5$ curve to show that
\[
Griff ^3(A^5) \neq 0 \neq Griff ^4(A^5).
\]
We write down his result
\begin{theorem}[Fakhruddin]
\textcolor{white}{.}
  \begin{itemize}
   \item $Griff ^3(A^5)$ and $Griff ^4(A^5)$ are not finitely generated.
   \item If $P$ is the generic Prym variety of dimension $g\geq 5$ then $Griff ^j(P)\neq 0$ for $3\leq j\leq g-1$.
   \item For the generic jacobian $J$ of dimension $g\geq 11$ we have $Griff ^{g-1}(J)\neq 0$.
   \item In arbitrary characteristic, we have for the generic curve $C$ of genus
$g\geq 3$ that $C-C^-$ is homologically but not algebraically equivalent to
zero.
  \end{itemize}
\end{theorem}
The arguments in \cite{Fk} are even strong enough to prove the result
of Ceresa in arbitrary characteristic. In fact, the result of Fakhruddin is in
one point better: $Griff _{(2)}^3(A^5) \neq 0 \neq Griff _{(2)}^4(A^5)$, where
\[
Griff ^r_{(i)}(\cdot):=\frac{Ch^r(\cdot)_{hom}\cap Ch^r_{(i)}(\cdot)}{Ch^r(\cdot)_{alg}\cap Ch^r_{(i)}(\cdot)}.
\]
Here $Ch^r_{(i)}(X)$ denotes the eigenspace
\[
Ch^r _{(i)}(X)  := \left \{ \alpha \in Ch^r(X): (m_{X})_*\alpha =m^{2n-2r+i}\alpha \right \} .
\]
We want to generalize in some sense the arguments of \cite{Fk}. We work with the
generic abelian fourfold $A^4$ with polarization of type $(1,2,2,2)$ and
consider the \emph{\bf higher Griff{}iths groups} $Griff ^{r,s}$ instead of the
classical ones (for the explicit definition look section at 1). Similar to
Fakhruddin we consider double covers of genus $7$ curves over genus $3$ curves.
The main difference is that we allow ramification points (Hurwitz formula
implies that there are $4$ ramification points). The key point is to use the
fact that $A^4$ is a generalized Prym variety associated to such a double cover.
This is a result in \cite{BCV}.\\
We also use ideas of \cite{Ike} about the theory of higher infinitesimal
invariants to prove the non-triviality of
\[
 Griff ^{3,2}(A^4).
\]
 
\vspace{.6cm}
This paper is organized as follows: Following \cite{Fk} we give the construction
of the degeneration of $A^4$. This means, we construct a
family $f:X\longrightarrow S$ of generalized Prym varieties such that the
classifying map $S\longrightarrow \mathcal A_4(1,2,2,2)$ is dominant. We
construct a relative cycle $Y/S$ on $X/S$ together with a subvariety $T\subset
S$ in such a way that we can give an explicit description of the embedding
$Y\vert _T\hookrightarrow X\vert _T$. This description will be given in the
second section. Then we show that some special components of $[Y]\in
H^2(S,R^4f_*\C)$ are non-trivial.\\
In the next section we show that the \emph{second} infinitesimal invariant
$\delta _2(\alpha )$ of $\alpha$ is non-trivial. Here $\alpha$ denotes the
component of $Y$ in $Ch^3_{(2)}(X/S)$ under the decomposition of Beauville (cf.
\cite{Be2} or \cite{DeMu}) for $Ch^3(X/S)$. Using this and our results from
section \ref{Cohm}, we can show that
\[
 0\neq [\alpha ]\in Griff^{3,2}(X/S).
\]
We can get a refined version of this result (cf. Theorem \ref{a4}):
\begin{theorem}\label{maintheorem}
For $s\in S$ generic we have
\[
0\neq [\alpha _s ]\in Griff ^{3,2}(A^4) ,
\]
where $A^4$ is the generic abelian variety of dimension $4$ with polarization of type $(1,2,2,2)$.
\end{theorem}

\subsubsection*{Acknowledgments}

The paper contains parts of my PhD thesis at the university on Mainz. I want to
thank my advisor Prof. Stefan M\"uller-Stach for his advice, without which the
thesis could not be realized. I also thank Prof. Fakhruddin for his
patience in answering a lot of emails that helped me to understood his work and
to improve my own work.


\section{Preliminaries and notations}

First we want to fix some notation. In all this section $f:X\longrightarrow S$ is smooth and projective and $X$, $S$ are smooth projective varieties.

\subsection*{The filtration of Saito and the $Z$-filtration}\label{Saitofil}
We recall the filtration $F_S$ defined by Saito \cite{Sa0} and its main
properties.\\
Let $F^0_SCh^p(X/S):=Ch^p(X/S)$ and
\[
F^{s+1}_SCh^p(X/S):= \sum _{Y, \Gamma , q}Im \left( \begin{xy}
 \xymatrix{
 F^s_SCh^{p+d-q}(Y/S) \ar[rr]^{\Gamma _*} && Ch^p(X/S)  }
\end{xy}
 \right ) ,
\]
where $Y$, $\Gamma$ and $q$ range over all following data:
\begin{itemize}
\item[(i)] $Y$ is projective smooth scheme, flat over $S$ and of relative dimension $d$,
\item [(ii)] $q\in \mathbb Z$ satisfy $p \leq q \leq p+d$,
\item [(iii)] $\Gamma \in Ch^q(Y\times _S X/S)$ is an algebraic cycle with the property
\[
\Gamma _*\left ( H_{dR} ^{2(p+d-q)-s}(Y_s) \right ) \subset F^{p-s+1}H_{dR}^{2p-s}(X_s) ;
\]
\end{itemize}
where $F$ is the  Hodge filtration on the De Rham cohomology. In case $S=Spec(\C )$ we write $F^sCh^p(X)$ instead of $F^s_{Spec(\C )}Ch^p(X/Spec(\C ))$.\\
We state some important properties of this filtration (cf. \cite{Sa0} and
\cite{Sa}):
\begin{satz}\label{standardarguments}
Assume $X/S$ is projective and smooth, then it holds:
\begin{itemize}
 \item [{\bf SF0.}] The filtration $F_S$ is stable under base extension and correspondences. In particular for $\alpha \in F_S^sCh^p(X/S)$ and $s_0\in S$ we have $\alpha _{s_0}\in F^sCh^p(X_{s_0})$.
\item [{\bf SF1.}] Inversely if $\alpha _{s_0}\in F^sCh^p(X_{s_0})$ for $s_0\in S$ the generic point, then there exist an open subset $U\subset S$ and an \'etale map $f:T\longrightarrow U\subset S$ s.t. $f^*(\alpha)\in F^s_TCh^p(T\times _SX/T)$.
\item [{\bf SF2.}] $F^{s+1}_SCh^p(X/S)=F^s_SCh^p(X/S)$ for all $s\geq p+1$.
\item [{\bf SF3.}] $F^1_SCh^r(X/S)=Ch^r(X/S)_{hom}:=\ker \left \{ Ch^r(X/S) \longrightarrow H^r(X,\O ^r_X) \right \}$.
\item [{\bf SF4.}]$F^2_SCh^r(X/S) \subset \ker \left \{ AJ^r_X :Ch^r(X/S)_{hom} \longrightarrow J^r(X) \right \}$ , where $J^r(X)$ is the intermediate jacobian and $AJ^r_X$ the Abel-Jacobi map.
\end{itemize}
\end{satz}

\subsubsection*{The $Z$-Filtration}
We can define an ascending filtration on the $F^s_SCh^r(X/S)$
\[
0\subset Z_0F^s_SCh^r(X/S) \subset Z_1F^s_SCh^r(X/S) \subset \cdots \subset Z_{n-r}F^s_SCh^r(X/S)= F^s_SCh^r(X/S) ,
\]
where $n$ is the relative dimension of $f:X\longrightarrow S$. This is done by
\[
Z_lF^s_SCh^r (X/S)= \sum _{Y,\Gamma } Im \left \{  \Gamma _*:F^s_SCh^{d-l}(Y/S) \longrightarrow Ch^r(X/S) \right \} ,
\]
where $Y/S$ ranges over all projective and smooth varieties of relative dimension $d$ and $\Gamma$ over $ Ch^{l+r}(Y\times _SX/S)$.\\
Similar to proposition (\ref{standardarguments}) one has:
\begin{satz}\label{Z0-F}
Let $X/S$ be as before.
\begin{itemize}
 \item [{\bf Z0.}] The $Z$ filtration on $F^s_SCh^r(X/S)$ is stable under base extension and correspondences. In particular, for $\alpha\in Z_lF_S ^sCh^r(X/S)$ and $s_0\in S$ we have $\alpha _{s_0}\in Z_lF^sCh^r(X_{s_0})$.
\item [{\bf Z1.}] Conversely, if $\alpha _{s_0}\in Z_lF ^sCh^r(X_{s_0})$ for $s_0\in S$ the generic point, then there exist an open subset $U\subset S$ and an \'etale map $f:T\longrightarrow U\subset S$ s.t. $f^*(\alpha ) \in Z_lF_S ^sCh^r(T\times _SX/T)$.
\item [{\bf Z2.}] For $Ch^r(X/S)_{alg}$ the subgroup of $Ch^r(X/S)$ which consist of cycles algebraically equivalent to zero holds
\begin{equation}
Ch^r(X/S)_{alg} = Z_0F^1_SCh^r(X/S). \label{z0f1}
\end{equation}
\end{itemize}
\end{satz}
The Beauville decomposition is compatible with the filtration of Saito:
\begin{satz}\label{mmm}
Let $X,S$ be smooth and connected $\C$-schemes and $f:X\longrightarrow S$ smooth and projective. Then
\[
\bigoplus _{i\geq s} Ch ^r_{(i)}(X/S) \subset  F^sCh^r(X/S) .
\]
\end{satz}
\begin{proof}
\cite{Mu}.
\end{proof}

\subsection*{The higher Griff{}iths groups}
They are defined as follows
\[
Griff ^{r,s}(X/S):=\frac {F^s_SCh^r(X/S)} { F^{s+1}_SCh^r(X/S)+Z_0F^s_SCh^r(X/S)}. 
\]
From the definitions of $F_S$ and $Z$ we get
\begin{align}
Griff ^{r,1} & \simeq \left ( \frac{Ch^r(X/S)_{hom}}{Ch^r(X/S)_{alg}} \right ) \left/ F^2_SCh^r(X/S)\right. \notag \\
 & =\frac{ Griff ^r(X/S)}{F^2_SCh^r(X/S)} , \notag
\end{align}
where $Griff ^r(X/S)$ stands for the classical Griff{}ths groups.

\subsection*{The Leray spectral sequence}

It is well known that the Leray spectral sequence
\[
E^{p,q}_2=H^{p}(S,R^qf_*\C ) \Rightarrow H^{p+q}(X,\C)
\]
degenerates in $E_2$ and in the case when $f$ is a relative abelian scheme the induced decomposition
\begin{equation}
H^r(X,\C) \simeq \bigoplus _{q} H^{r-q}(S,R^qf_*\C ).\label{cano}
\end{equation}
is canonical: We can identify $H^{r-q}(S,R^qf_*\C )$ with the subspace of $H^r(X,\C )$ on which, for all $m\in \mathbb
Z$, the multiplication maps $(m_{X/S})_*$ act as multiplication by $m^{2n-q}$. This shows
\begin{satz}
Let $f:X\longrightarrow S$ be a relative smooth and projective abelian scheme
and $cl : Ch^r(X/S) \longrightarrow H^{2r}(X,\mathbb Q)\subset H^{2r}(X,\C)$ be
the cohomology class map. Then we have
\[
cl \left ( Ch^r_{(s)}(X/S)  \right ) \subset H^{s}(S,R^{2r-s}f_*\C ) .
\]
\end{satz}

\subsection*{Higher infinitesimal invariants}

We assume now that the fibers of $f:X\longrightarrow S$ are polarized abelian varieties. We have an exact sequence
\[
\begin{xy}
\xymatrix{
0 \ar[r]^{} & f^*\Omega ^1_S \ar[r]^{} & \Omega ^1_X \ar[r]^{} & \Omega ^1_{X/S} \ar[r]^{} & 0.
}
\end{xy}
\]
We define the following subsheaves of $\O ^r_X$: $L^p_S(r):=L^p_S\O ^r_X:= f^*\O ^p _S \wedge \O ^{r-p}_X $ with graded pieces $Gr_L^p(r):=Gr_L^p\O ^r_X = f^*\O ^p_S\otimes \O^{r-p}_{X/S}$. It is clear that
\[
\O _X^r =L^0_S(r) \supseteq L^1_S(r) \supseteq \cdots L^{r}_S(r) \supseteq L^{r+1}_S(r)=0.
\]
With this we can give a filtration of the coherent sheaves $R^pf_*\O ^r_X$ through
\[
L^pR^lf^*\O ^r_X:=Im \left ( R^lf_*L^p_S\O^r_X \longrightarrow R^lf_*\O ^r_X \right ) .
\]
We have a spectral sequence (cf. Section 5.2 of \cite{Vo} II), the holomorphic Leray spectral sequence:
\begin{equation}
\E ^{p,q}_1(r,X):= R^{p+q}f_*(Gr_L^p(r)) \Rightarrow R^{p+q}f_*\O ^r_X \label{S-S} 
\end{equation}
with
\[
\E ^{p,q}_{\infty}(r,X)=Gr_L^pR^{p+q}f_*\O ^r_X.
\]
The differential $d_1 :\E ^{p,q}_1(r,X) \longrightarrow \E ^{p+1,q}_1(r,X)$ is the connecting homomorphism
$R^{p+q}f_*(Gr_L^p(r))\longrightarrow R^{p+q+1}f_*(Gr^{p+1}_L(r))$ induced by
\[
\begin{xy}
\xymatrix{
0 \ar[r]^{} & Gr_L^{p+1}(r) \ar[r]^{} & L^p_S(r) / L^{p+2}_S(r) \ar[r]^{} & Gr^p_L(r) \ar[r]^{} & 0 .
}
\end{xy}
\]
Using the projection formula we can see that $\E^{p,q}_1(r,X)\simeq \O ^p_S\otimes R^{p+q}f_*\O^{r-p}_{X/S}\simeq \O
^p_S\otimes \mathcal H^{r-p,r+p}$, where $\mathcal H^{i,j}$ are the Hodge bundles. It can be proved (cf. \cite{Vo} II,
pg. 139) that $d_1$ can be identified with the map $\overline{\nabla}$ induced by $\nabla$, the Gauss-Manin connection
at the $p$-th step of:
\[
\begin{xy}
\xymatrix{
0 \ar[r] & Gr _F^r\mathcal H^{r+q} \ar[r]^{\overline{\nabla} \qquad} &  \O^1_S \otimes Gr^{r-1}_F\mathcal H^{r+q} \ar[r]^{\overline{\nabla}} & \O^2_S\otimes Gr_F^{r-2}\mathcal H^{r+q} \ar[r] & \cdots 
}
\end{xy}
\]
Here $F$ stands for the Hodge filtriation. Ikeda \cite{Ike} has shown that this spectral sequence degenerates in $\E_2$. Now if $cl$ denotes the composition
\[
\begin{xy}
\xymatrix{
Ch^r(X/S) \ar[r]^{cl_{\C}} & H^r(X,\O ^r_X) \ar[r] & H^0(S, R^rf_*\O ^r_X) }
\end{xy}
\]
it is possible to show (cf. \cite{Ike} Lemma 2.6) that
\[
cl \left ( F^p_SCh^r(X/S) \right ) \subset H^0(S, L_S ^pR^rf_*\O ^r_X),
\]
which leads to the following important definition:
\begin{definition}
For an algebraic cycle $\alpha \in F^s_S(X/S)$ we denote by $\delta _s(\alpha)$ the image of $cl(\alpha)$ under the map
\[
H^0(S,L_S ^sR^rf_*\O ^r_X)\longrightarrow H^0(S, Gr^s_L R^rf_*\O ^r_X)
\]
and call it the {\bf higher infinitesimal invariant of $\alpha$}.
\end{definition}

\subsection*{Moduli of double covers}

Remember the definition of \emph{stable curves}:
\begin{definition}
A genus $g$ curve $D$ is called a stable curve if the following conditions hold:
\begin{itemize}
\item $D$ is connected and reduced,
\item $D$ has only ordinary double points as singularities,
\item if $K$ is a smooth rational component of $D$ then $K$ intersects the other components in at least 3 points.
\end{itemize}
\end{definition}
Now we fix a natural number $n\geq 3$. Let $\overline{\mathcal M}_3^{(n)}$ be the  moduli space of stable genus $3$
curves with a level $n$-structure\footnote{For a definition of level $n$-structure we refer to \cite{Mo}.}. We define
now
$\overline{\mathcal R}(3,2)(n)$ through the following pullback diagram
\[
\begin{xy}
\xymatrix{
\overline {\mathcal R}(3,2)(n) \ar[r]^{} \ar[d]^{} & \overline{\mathcal M} ^{(n)}_3 \ar[d]^{} \\
\overline{\mathcal R}(3,2) \ar[r]^{} & \overline{\mathcal M} _3 }
\end{xy}
\]
where $\mathcal R(3,2)$ and $\overline{\mathcal R}(3,2)$ are defined as in \cite{BCV}. In particular we see that
$\overline{\mathcal R}(3,2)(n):=\overline{\mathcal M} ^{(n)}_3\times _{\overline{\mathcal M}_3} \overline{\mathcal
R}(3,2)$. $\overline{\mathcal R}(3,2)(n)$ is smooth because it is \'etale over the manifold $\overline{\mathcal M}
^{(n)}_3$.
It is known that there exists a universal family $\Gamma ^{(n)}_3
\longrightarrow \overline{\mathcal M}^{(n)} _{3}$. Let $\Gamma ^{(n)}$ be the pullback of $\Gamma ^{(n)}_3$ under the
projection map $\overline{\mathcal R}(3,2)(n)\longrightarrow \overline{\mathcal M}_3^{(n)}$.
\section{Construction of the family of curves and of the generalized Prym variety}\label{Deg}

In this section we want to indicate how to modify the arguments in \cite{Fk} to get our corresponding version of the degeneration of the generic element in $\mathcal A_4(1,2,2,2)$.\\
We follow the paper \cite{Fk} for our construction of the degeneration. The main difference with Fakhruddin's
construction is that in \cite{Fk} he considers only \'etale covers and here we allow ramification in $4$ points. This
difference is reflected in switching from $\overline{\mathcal M}_3^{(n)}$ to $\mathcal R(3,2)(n)$.

\subsection*{Description of a family of curves}\label{fa}

We fix a genus $2$ curve $C$, an elliptic curve $E$ and $4$ points $p_0, p_1, p_2, p_3$ on $E$. We will assume that the
following condition holds:
\[
\textbf{Condition:} \text{ There are no Hodge classes of type $(1,1)$ in $H^1(C,\Q)\otimes H^1(E,\Q )$.}
\]
This condition is satisfied by choosing $C$ and $E$ generic. For each pair $(x,y)\in C\times (E-\{ p_0,\ldots ,p_3 \} )$ let $D_{(x,y)}$ be the curve obtained by glueing $C$ and $E$ through identification of the points $x$ and $y$.
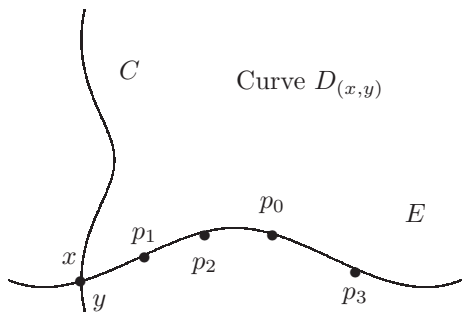
\[
\begin{xy}
(-30,-6)*{}; (30,-6)*{} **\crv{(-20,-10) & (0,12) & (20, -10)};
(-20,-10)*{}; (-20,30)*{} **\crv{(-22,0) & (-10,10) & (-22,20)};
(24,3)*{ E}; (-14,22)*{C};(-18,-9)*{y}; (-22,-3)*{x}; (-20.5,-6.2)*{\bullet}; 
(10,20)*{\txt{Curve $D_{(x,y)}$}};
(-12, 0)*{p_1}; (-12,-3)*{\bullet}; (-4, -4)*{p_2};(-4,0)*{\bullet}; (5, 4)*{p_0};(5,0)*{\bullet}; (16,-8)*{p_3}; (16,-5)*{\bullet}; 
\end{xy}
\]
It follows that $D_{(x,y)}$ is a genus $3$ stable curve. We obtain in this way a family of stables curves $\mathcal C \longrightarrow C\times (E-\{ p_0,\ldots ,p_3 \})$ with a relative divisor $B:=p_0+p_1+p_2+p_3$. By choosing a level $n$-structure and a non-trivial line bundle $\mathcal L$ with $\mathcal L^{\otimes 2}\simeq \mathcal O (B)$ we get an injection
\[
h:C\times (E-\{ p_0, p_1, p_2, p_3 \}) \longrightarrow S_1 \quad \text{with} \quad
h^*(\Gamma ^{(n)})=\mathcal C,
\]
where $S_1$ is the open subspace of $ \overline{\mathcal R}(3,2)(n)$ consisting of treelike curves. Let $T_1=h(C\times
\left ( E-\{ p_0,\ldots ,p_3 \})\right ) \subset S_1$ (so the restriction of the family to $T_1$ consist of curves of
the form $D_{(x,y)}$) and $\Gamma _1$ the restriction of $\Gamma ^{(n)}\to S_1$. Since our family consists of curves
with marked points we have, in a natural way, four sections, which restricted to $T_1$ are given by $p_0,\ldots , p_3$. We
will denote these sections again with $p_0,\ldots ,p_3$.

\subsection*{Construction of the family of double covers}\label{Kons-Uber}

In this part we will fix one of these sections, say $p_0$. We refer to the paper \cite{Kl} of Kleiman as a reference
for details of some constructions here.\\
For a given $S_1$-morphism $g: T\longrightarrow S_1$ consider the following functors
\[
\mathcal Pic (\Gamma _1/S_1)(T):= \left \{  \begin{array}{c}
\text{Classes of isomorphisms } \\
\text{ of line bundles } \mathcal L\text{ on } \\
\Gamma _1\times _{S_1}T \textbf{ rigidified} \\
\text{ along } p_0
                                            \end{array}
\right \}
\]
and
\[
\mathcal Div ^m (\Gamma _1/S_1)(T):= \left \{ \begin{array}{c}
\text{ Relative effective divisors } \\
D \text{ of } \Gamma _1\times _{S_1}T \text{ of degree } m      \end{array}
\right \} .
\]
The meaning of \emph{$\mathcal L$ rigidified  along $p_0$} is (cf. \cite{BL} p. 597):
$$(p_0\circ g, 1_T)^* \mathcal L \simeq \mathcal O _{T}.$$
The first functor is represented by a smooth algebraic space ${ \bf Pic } _{\Gamma _1/S_1}$ of finite type over $S_1$. This space is a group scheme (cf. \cite{BLR} pg. 204).\\
Now we can see why it was necessary to have sections of $S_1 \longrightarrow \Gamma _1$.\\
Since $\Gamma _1 \longrightarrow S_1$ has curves as fibres (and is flat and projective) we have that the functor of
relative effective divisors with fibres of degree $m$ is represented by a scheme (here is necessary $m \geq 1$ too, cf.
\cite{Kl} pg. 24). Let ${\bf Div} ^m_{\Gamma _1 /S_1}$ be this scheme.\\
Through the $4$ sections we get in a natural way a divisor $B$ on $\Gamma _1=\Gamma _1\times _{S_1}S_1$ defined as
\[
B:=Im \{ p_0+p_1+p_2+p_3 \} \subset \Gamma _1
\]
and then a $S_1$-morphism $b: S_1\longrightarrow { \bf Div }^{4} _{\Gamma _1/S_1}$ (i.e. a section of ${ \bf Div }^{4} _{\Gamma _1/S_1}\longrightarrow S_1$). The composition 
\[
{\bf A}^4_{\widetilde{\Gamma _1}/S_1}\circ b \in Hom _{S_1}(S_1, { \bf Pic } _{\Gamma _1/S_1})
\]
defines now a line bundle $\mathcal L_0$ on $\Gamma _1$. Let $\phi$ be the following composition
\[
\begin{xy}
 \xymatrix{
 { \bf Pic } _{\Gamma _1/S_1}\ar[r]^{2} & { \bf Pic } _{\Gamma _1/S_1} \ar[r]^{\otimes {\bf \mathcal L_0}^{-1}} & { \bf Pic } _{\Gamma _1/S_1} }
\end{xy}
\]
and define
\[
S_2:=\ker (\phi ) -(Zero) \hookrightarrow { \bf Pic } _{\Gamma _1/S_1} .
\]
We get a $S_1$-morphism
\[
f:S_2 \hookrightarrow {\bf Pic }_{\Gamma _1/S_1}\longrightarrow S_1.
\]
On $\Gamma _2:=\Gamma _1\times _{S_1}S_2$ we then have a line bundle
$\mathcal M$ with the property $\mathcal M ^{\otimes 2}\simeq \mathcal L_0$. We denote by the same symbol $\mathcal L_0$ also the pullback  of $\mathcal L_0$ to $\Gamma _2$ by $f_2$.\\
From this we get a family of double covers
\[
U: \widetilde {\Gamma}_2 \longrightarrow \Gamma _2
\]
(with $B$ as ramification divisor). Let $T_2'$ be the subset of $T_2=f_2^{-1}(T_1)$ such that
\[
\widetilde{\Gamma}_2 \vert _{T_2'} \longrightarrow T_2'
\]
consists of curves of the form
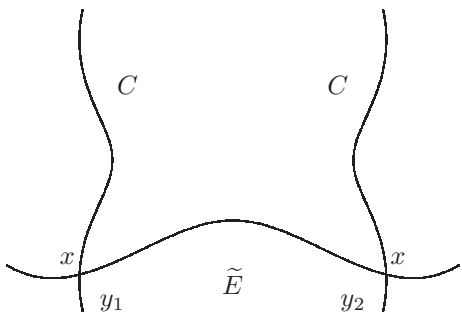
\[
\begin{xy}
(-30,-4)*{}; (30,-4)*{} **\crv{(-20,-10) & (0,14) & (20, -10)};
(20,-10)*{}; (20,30)*{} **\crv{(22,0) & (10,10) & (22,20)};  (14,20)*{C}; 
(-20,-10)*{}; (-20,30)*{} **\crv{(-22,0) & (-10,10) & (-22,20)}; (-14,20)*{C};
(0,-6)*{ \widetilde E}; (-16,-9)*{ y_1}; (-22,-3)*{x};
(16,-9)*{y_2}; (22,-3)*{x}; 
\end{xy}
\]
where:
\begin{itemize}
\item $\widetilde E$ is a ramified double cover of $E$ with ramification in $p_0, p_1, p_2, p_3$ (in particular $\widetilde E$ is a genus $3$ curve),
\item $y_1 \neq y_2$ are the elements of $U^{-1}(y)$ (in particular the $4$ ramification points are in $E-\{ y\}$),
\end{itemize}
Let $T_3$ be the connected component of $T_2$ which contains $T_2'$. From this construction it follows in particular that $f_2\vert _{T_3}:T_3\longrightarrow T_1$ is an isomorphism.\\
Let $S_3$ be the connected, Zariski-open subset of $S_2$ that contains
$T_3$ and
\[
\widetilde {\Gamma}_3 :=\widetilde {\Gamma}_2 \vert _{S_3}
\longrightarrow S_3
\]
consists of  \emph{treelike} curves. We use the same notation for the points in $\widetilde E$ determined by $p_0,...,p_3\in E$ under $U$. Let $\widetilde T_3:= C\times \left(\widetilde E-\{ p_0,...,p_3\} \right )$.\\
\begin{minipage}{10cm}
In some sense we want to \emph{lift} the parameters from $T_3$ to $\widetilde T_3$. We make a base change of $\widetilde
{\Gamma}_3\to S_3$ through the map $\xymatrix{ \widetilde T_3 \ar[r]^{2:1} & T_3 \ar@{^(->}[r] & S_3}$. Doing this
we get a fibre product diagram as on the right.
\end{minipage}
\hfill
\begin{minipage}[cl]{4cm}
\[
\xymatrix{ \widetilde {\Gamma} _4 \ar[d] \ar[r] & \widetilde {\Gamma} _3  \ar[d] \\
\widetilde T_4 \ar[r] & S_3
}
\] 
\end{minipage}
Now we want to fix a base point of the curves to get an inclusion in the Jacobian. Therefore let $D$ be the divisor $\{
y+iy\in (\widetilde {\Gamma} _4)_{(x,y)} : y\in \widetilde E \}$. This divisor meets a general fibre in 2 points and
hence we get a generically finite surjective morphism $\widetilde T_5\to \widetilde T_4$ and $2$ effective divisors
$D_1, D_2$ on $\widetilde {\Gamma} _5:= \widetilde {\Gamma} _4\times _{\widetilde T_4} \widetilde T_5$ meeting a general
fibre in only one point and such that: if $g:\widetilde {\Gamma _5}\longrightarrow \widetilde {\Gamma}_4$ is the natural
projection then $g^*(D)=D_1+D_2$. The geometric meaning is that we have \emph{marked} (or choosed) one of the
copies of $C$ (we can say the left copy) by fixing $D_1$. The unique points on each fibre determined by $D_1$ give a
section of $\widetilde {\Gamma}_5\longrightarrow \widetilde T_5$. Let $f:S\longrightarrow  S_5$ be an \'etale map such
that $\widetilde {\Gamma}:=\widetilde {\Gamma}_5\times _{S_5}S\to S$ has a section $\sigma : S\longrightarrow \widetilde
{\Gamma}$ extending the section of $\widetilde {\Gamma}_5\longrightarrow \widetilde T_5$.\\
\begin{minipage}{10cm}
Let $T$ be the corresponding component of $f^{-1}(\widetilde T_5)$ such that the family $\widetilde {\Gamma} \vert _T
\longrightarrow T$ consist of curves of the form as in the right picture. The families $\widetilde {\Gamma}
\longrightarrow S$, $\widetilde
{\Gamma}\vert _T \longrightarrow T$ and the section $\sigma$ will be used to construct a non-trivial cycle on the
generic abelian fourfold of type $(1,2,2,2)$ (cf. Theorem \ref{a4}).
\end{minipage} \hfill
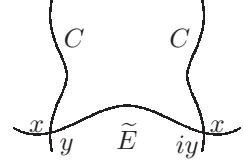
\begin{minipage}{4.5cm}
\[
\begin{xy}
(-15,-2)*{}; (15,-2)*{} **\crv{(-10,-5) & (0,7) & (10, -5)};
(10,-5)*{}; (10,15)*{} **\crv{(11,0) & (5,5) & (11,10)};  (7,10)*{C};
(-10,-5)*{}; (-10,15)*{} **\crv{(-11,0) & (-5,5) & (-11,10)}; (-7,10)*{C};
(0,-3)*{ \widetilde E}; (-8,-4.5)*{y}; (-12,-1.7)*{x};
(8,-4.5)*{iy}; (12,-1.7)*{x};
\end{xy}
\]
\end{minipage}
\subsection{Generalized Prym variety and cycle}

As before  ${\bf Pic}_{\widetilde{\Gamma} /S}$ exists (since $\widetilde {\Gamma} \longrightarrow S$ has a section). We
will follow the arguments of \cite{Fk} pg. 113 for our construction of the (generalized) Prym variety and cycle.\\
Let us introduce some notation first.
\begin{itemize}
\item ${\bf Pic}^0_{\widetilde{\Gamma} /S}$ is the open subspace of ${\bf Pic}_{\widetilde{\Gamma} /S}$ representing the
functor of line bundles which have degree zero on each component of each fibre.
\item ${\bf Pic}^{(j)}_{\widetilde{\Gamma} /S}$, $j\in \mathbb Z$ is the open subspace of ${\bf Pic}_{\widetilde{\Gamma} /S}$ corresponding to line bundles of total degree $j$ on each fibre.
\item $H\subset {\bf Pic}^{(0)}_{\widetilde{\Gamma} /S}$  is the closure of the connected component of the identity.
\item $P_c:={\bf Pic}^{(0)}_{\widetilde{\Gamma} /S}/H$.
\end{itemize}
It is clear from the definitions that ${\bf Pic}^0_{\widetilde{\Gamma} /S} \subset {\bf Pic}^{(0)}_{\widetilde{\Gamma}
/S}$. We also have that
\[
\phi : {\bf Pic}^{0}_{\widetilde{\Gamma} /S} \longrightarrow {\bf Pic}^{(0)}_{\widetilde{\Gamma} /S} \longrightarrow P_c
\]
is an isomorphism, because it is on the fibres. Let $\Gamma _0$ be the open (and dense) subspace of $\widetilde
{\Gamma}$ at which $\pi :\widetilde {\Gamma}\longrightarrow S$ is smooth. There is a natural morphism $\gamma _1 :\Gamma
_0 \longrightarrow {\bf Pic}^{(1)}_{\widetilde{\Gamma} /S}$ and we can then define $\gamma _0:\Gamma _0 \longrightarrow
{\bf Pic}^{(0)}_{\widetilde{\Gamma} /S}$ by $\gamma _0(p):= \gamma _1(p)-\gamma _1(\sigma (\pi(p)))$.
Let $\gamma :\widetilde {\Gamma}\longrightarrow {\bf Pic}^{0}_{\widetilde{\Gamma} /S}$ be the following morphism
\[
\xymatrix{
 \Gamma _0 \ar[r]^{\gamma _0\quad } & {\bf Pic}^{(0)}_{\widetilde{\Gamma} /S} \ar@{->>}[r] & P_c \ar[r]^{\phi ^{-1}} & {\bf Pic}^{0}_{\widetilde{\Gamma} /S}
}
\]
Because of the normality of $\widetilde {\Gamma}$ we can extend $\gamma$ to a morphism $\gamma :\widetilde {\Gamma}\longrightarrow {\bf Pic}^{0}_{\widetilde{\Gamma} /S}$. The involution on $\widetilde {\Gamma}$ gives another involution $i$ on ${\bf Pic}^0_{\widetilde{\Gamma} /S}$.
\begin{definition}
The abelian scheme
\[
X=Prym(\widetilde{\Gamma}/\Gamma):=Im \left ( 1- {i} :{\bf Pic}^0_{\widetilde{\Gamma} /S} \longrightarrow {\bf Pic}^0_{\widetilde{\Gamma} /S}
\right )
\]
is called the {\bf relative generalized Prym variety}.
\end{definition}
$\pi
:X\longrightarrow S$ is a (non principaly) polarized abelian scheme of
relative dimension $4$. We have a natural cycle on $X$, namely
\[
Y:=\left ( (1-i)\circ \gamma \right )(\widetilde{\Gamma})
\hookrightarrow X
\]
This is a subvariety of relative codimension $3$.

\section{Description of the embedding $Y\vert _T\hookrightarrow X\vert _T$} \label{NotaZykel}

This section is very important because it contains the explicit description of our cycle and this will be used to describe its cohomology class.\\
Observe that for $(x,y)\in T\simeq C\times (\widetilde E-\{ p_0,\ldots ,p_3 \})$ we have $X_{(x,y)}\simeq J \times P$, where $P:=Prym(\widetilde E /E)$ is the corresponding generalized Prym variety and $J=J(C)$. Let $\widetilde J$ be the jacobian of $\widetilde E$ ($\widetilde J$ has dimension 3). This is an easy consequence of how $(1-i)$ acts on each of the components of
\[
Pic ^0(\widetilde {\Gamma}_{(x,y)}) \simeq J\times \widetilde J\times J .
\]

To describe $Y\vert _T \hookrightarrow X\vert _T$ we need first to calculate $\gamma _{(x,y)} :\widetilde{\Gamma}_{(x,y)}
\hookrightarrow Pic ^0(\widetilde {\Gamma}_{(x,y)})$ and then apply $1-i$ to the image.

\subsubsection*{Step 1: $\gamma _{(x,y)} :\widetilde{\Gamma}_{(x,y)} \hookrightarrow Pic ^0(\widetilde {\Gamma}_{(x,y)}) \simeq J \times \widetilde J\times J$}\label{Ein-1}

\begin{minipage}{9cm}
Let $\bar O$ be the point defined by $\sigma \vert _T$. We analyse now the image of a point $z\in C$. We can see that
\[
\gamma _{(x,y)}(z)= [z-\bar O] \times 0 \times 0 =[z-x]\times 0\times 0.
\]
In particular
\[
\gamma _{(x,y)}(y_1)=\gamma _{(x,y)}(x)=[x-x]\times [y-y] \times 0.
\]
Since this description should be compatible in the points $x$ and $y=\bar O$ it is
necessary that the map $\widetilde E\hookrightarrow Pic^0(\widetilde {\Gamma}_{(x,y)})$ has the following form:
$z\mapsto 0 \times [z-y]\times 0$.
\end{minipage} \hfill
\begin{minipage}{5.3cm}
\[
\begin{xy}
(-30,-4)*{}; (30,-4)*{} **\crv{(-20,-10) & (0,14) & (20, -10)};
(20,-10)*{}; (20,30)*{} **\crv{(22,0) & (10,10) & (22,20)};  (14,20)*{C}; 
(-20,-10)*{}; (-20,30)*{} **\crv{(-22,0) & (-10,10) & (-22,20)}; (-14,20)*{C}; (-16,10)*{\bullet}; (-13,10)*{z};
(0,-6)*{ \widetilde E}; (-16,-9)*{ y}; (-22,-3)*{x}; (-20.5,-5.5)*{\bullet}; (-16,0)*{\bar O};
(16,-9)*{iy}; (22,-3)*{x};
\end{xy}
\]
\end{minipage}
If we continue the argument in this way we get the following result:
\begin{align}
C & \hookrightarrow J\times \widetilde J \times J  : \quad z\mapsto [z-x]\times 0\times 0 \notag \\
\widetilde E & \hookrightarrow J\times \widetilde J \times J : \quad z\mapsto 0\times [z-y]\times 0 \notag \\
C & \hookrightarrow J\times \widetilde J\times J : \quad z\mapsto 0\times [iy-y]\times [z-x] \notag
\end{align}

\subsubsection*{Step 2: $(1-i)\circ \gamma _{(x,y)} :\widetilde{\Gamma}_{(x,y)} \hookrightarrow Pic ^0(\widetilde 
{\Gamma}_{(x,y)}) \longrightarrow  J\times P$}\label{Ein-2}
We apply here the map $(1-i)$ to the above formulas, which leads to:
\begin{itemize}
\item [1)] $C\hookrightarrow J\times P$ is given by $z\mapsto [z-x]\times 0$,
\item [2)] $\widetilde E\hookrightarrow J\times P$ is given by $z\mapsto 0\times [y-z]$,
\item [3)] $C\hookrightarrow J\times P$ is given by $z\mapsto [x-z]\times
[2iy-2y]$.
\end{itemize}
This describes the embedding $Y_{(x,y)} \hookrightarrow X_{(x,y)}$ for each $(x,y)\in T$. We see that the above description makes sense over the points of $p_0,\ldots , p_3 \in \widetilde E$ and therefore  for all points of $C\times \widetilde E$. We can then extend the family and take the total space $X\vert _T$ as $(C\times \widetilde E)\times (J\times P)$.

\subsubsection*{The embedding $Y\vert _T \hookrightarrow X\vert _T$}\label{Einbe}

With the above fibrewise description we can give a description (now global) of $Y\vert _T \hookrightarrow X\vert _T$:
\begin{itemize}
\item [1)] $(C\times \widetilde E)\times C \hookrightarrow (C\times \widetilde E)\times (J \times P)$ is given through
$$(x_1,y,x_2) \mapsto (x_1,y, [x_2-x_1], 0),$$
\item [2)] $(C\times \widetilde E)\times \widetilde E \hookrightarrow (C\times \widetilde E)\times (J \times P)$ is
given through $$(x,y,z) \mapsto (x,y, 0,[y-z]),$$
\item [3)] $(C\times \widetilde E)\times C \hookrightarrow (C\times \widetilde E)\times (J\times P)$ is given through
$$(x_1,y,x_2) \mapsto (x_1,y, [x_1-x_2], [2iy-2y]).$$
\end{itemize}
This description will be a key point for the calculation of the cohomology class of our cycle. We are interested in the component (under the K\"unneth decomposition) in $H^2(\widetilde E\times C)\otimes H^{4}(J\times P)$ or more explicity in
$$\left ( H^1(\widetilde E)\otimes H^1(C) \right )\otimes \left ( H^{1}(J)\otimes H^3(P) \right ) .$$
\section{Cohomology class of $Y$}\label{Cohm}

We will see that the cohomology class of $Y$ is non-trivial. For this it is enough to show that the cohomology class $[Y\vert _T]$ is non-trivial.

\subsection*{The cohomology class of $C\times \widetilde E\times Y_{(x,y)}$ in $C\times \widetilde E\times X_{(x,y)}$}

The codimension of $C\times \widetilde E\times Y_{(x,y)}$ in $C\times \widetilde E\times X_{(x,y)}$ is $3$. We will use the above description of the embedding to show that the cohomology class of $C\times \widetilde E\times Y_{(x,y)}$ as element of $H^{6}((C\times \widetilde E)\times X_{(x,y)}, \Q )$ is non-trivial. Moreover that its component in 
\[
\left ( H^1(C)\otimes H^1(\widetilde E)\right ) \otimes \left ( H^{1}(J) \otimes H^{3}(P) \right )
\]
is non-trivial. To begin with, observe that the only non-trivial contribution in $\left ( H^1(C)\otimes H^1(\widetilde
E)\right ) \otimes \left ( H^{4}(J\times P) \right )$ comes from  $3)$ of the description of the embedding.\\
The cycle in $1)$ is given as a product of the following $2$ cycles:
\begin{itemize}
\item $Z_1$: The (codimension $1$) cycle $C\times C \hookrightarrow C\times J$, $(x_1,x_2)\mapsto (x_1, [x_2-x_1])$.
\item $Z_2$: The (codimension $2$) cycle $\widetilde E \hookrightarrow \widetilde E\times P$, $y \mapsto (y,[2iy-2y])$.
\end{itemize}
\begin{lemma}
One has
\begin{itemize}
\item $0\neq [Z_1] \in H^1(C)\otimes H^{1}(J)$ and
\item $0\neq [Z_2] \in H^1(\widetilde E) \otimes H^{3}(P)$
\end{itemize}
\end{lemma}
\begin{proof}
We prove only the assertion about $[Z_2]$, the proof for $[Z_1]$ being similar. Let $[Z_2]^{1,3}$ be the component of $[Z_2]$ in $H^1(\widetilde E)\otimes H^3(P)$.  If $J^2(P)$ denotes the intermediate jacobian, then 
\[
J^2(P)\simeq H^{1,2}(P)/H^{3}(P,\mathbb Z)\simeq H^0(P, \O ^1_P)^*/H_1(P,\mathbb Z)=:Alb(P)\simeq P
\]
and from a result of Griff{}iths (cf. \cite{Vo} I, pg. 294) we known that the following map is holomorphic:
\begin{align}
AJ^{Z_2}_P: \widetilde E &\longrightarrow J^2(P)\simeq P \notag \\
z&\mapsto AJ^2_P\left( (Z_{2})_{z}-(Z_2)_{\bar O} \right )=[z-\bar O] \quad \text{(last using the identification $J^2(P)\simeq P$)}.\notag
\end{align}
\begin{multicols}{2}
Now from the universal property of the Jacobian $\widetilde J$, the map $AJ^{Z_2}_P$ factors as in the right diagram.
In according with a proposition in \cite{Vo} I, pg. 291, the morphism $\psi _{Z_2}$ of complex tori is induced by the
morphism $[Z_2]^{1,3}\in H^1(\widetilde E,\mathbb Z)\otimes H^{3}(P,\mathbb Z) \label{CompKun}$.
\columnbreak
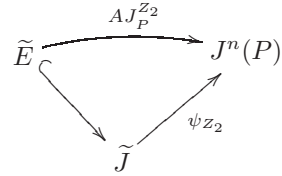
\[
\begin{xy}
\xymatrix{ \widetilde E \ar@/^/[rr]^{AJ^{Z_2}_P} \ar@{^(->}[rd]^{} & & J^n(P) \\
& \widetilde J \ar[ru]_{\psi _{Z_2}} &
   }
\end{xy}
\]
\end{multicols}
Obviously the map $AJ^{Z_2}_P$ is non-trivial and therefore, from the commutativity of the diagram, one has that $\psi
_{Z_2}$ is non-zero and in particular $[Z_2]^{1,3}\neq 0$.
\end{proof}
Let $a$ be this non-trivial component of $[Y_{\vert T}]$.

\subsection*{The primitive part of $a$ in \\
$\left ( H^1(C)\otimes H^1(\widetilde E) \right ) \otimes \left ( H^{1}(J)\otimes H^{3}( P) \right )$ is  non-trivial}\label{P-a}

We introduce the following notation:
\begin{align}
H(i,j): & = \left ( H^1(C)\otimes H^1(\widetilde E) \right ) \otimes \left ( H^i(J) \otimes H^j( P) \right ) \notag \\
H(i):& =\left ( H^1(C)\otimes H^1(\widetilde E)\right ) \otimes H^i(J \times P) \notag  
\end{align}
Let $\mathcal L$ be a relative ample line bundle on $X$. From the canonical decomposition (\ref{cano}) we can define the
image $L$ of $c_1(\mathcal L)\in H^2(X,\C)$ in $H^0(S,R^2f_*\C )$. With the above notation and the hard Lefschetz
Theorem we get a conmutative diagram
\[
\begin{xy}
\xymatrix{
  {H(2)} \ar[d]_{L} \ar[r]^{L^{2}}_{\simeq} & {H(6)} \\
  H(4)\ar[r]^{id}_{\simeq}  & H(4) \ar[u]_{L}
  }
\end{xy}
\]
and a decomposition
\[
H(4)\simeq ( L^{2}\mathbb P_0 \oplus L \mathbb P_2 ) \oplus \mathbb P_4 .
\]
Here $\mathbb P_j$ is a fiber of the local system $(R^jf_*\C)_{prim}\vert _T$.\\
{\bf Warning:} $H(4)$ is a direct sumand of $H^{6}(X \vert _T , \mathbb Q)$, thus
\[
0\neq a\in H(1,3) \subset H^{6}(X \vert _T , \mathbb Q).
\]
\begin{satz}
If $a'$ is the primitive part of $a$ in $H(1,3)$ then $a'\neq 0$.
\end{satz}
\begin{proof}
Since $L^2 \mathbb P_0 \oplus L \mathbb P_2 \subset LH(2)$ it is enough to show that $a$ is not contained in the image of
\[
L :H(2) \longrightarrow H(4).
\]

Remember that the polarization $L$ is the product of the polarizations of $J$ and $P$, this means $L=L_J + L_P$, where
\[
L_J \in H^0(C\times \widetilde E) \otimes (H^2(J)\otimes H^0(P))
\]
is the (principal) polarization of $J$ and
\[
L_P \in H^0(C\times \widetilde E) \otimes (H^0(J)\otimes H^2(P))
\]
is the (non-principal, but of type $(1,2)$) polarization of $P$.\\
From the hard Lefschetz Theorem we have a couple of isomorphisms:
\[
L_J : H(1,1) \simeq H(3,1), \quad L^2_J: H(0,2) \simeq H(4,2)
\]
and from these the following maps are injective:
\begin{align}
L_J & : H(1,1) \longrightarrow H(3,1) \label{LJ} \\
L_J & : H(0,2) \longrightarrow H(2,2) \notag
\end{align}
In a similar way we get injective maps:
\begin{align}
L_P & : H(1,1) \longrightarrow H(1,3) \notag \\
L_P & : H(2,0) \longrightarrow H(2,2) \notag
\end{align}
Assume that $\alpha \in H^2$ satisfies $L \alpha =a$. Let $\alpha ^{1,1}$ be its component in $H(1,1)$. Using
(\ref{LJ}) we see that it is necessary that $\alpha ^{1,1}=0$. Using this and the above other injections we get
\[
a=L\alpha = L_J\alpha + L_P \alpha \in H(4,0)\oplus H(2,2) \oplus H(0,4)
\]
but since $0\neq a\in H(1,3)$ we get a contradiction. Thus we conclude that $a$ is not in $L H(2)$.
\end{proof}
\begin{folgerung}
If $T=C\times \widetilde E$ and $X_0=J\times P$ then the Hodge component of $a'$ in
\[
H^0(T,\O ^2_T)\otimes H^3(X,\O ^1_{X_0}) _{prim}
\]
is non-trivial.
\end{folgerung}
\begin{proof}
After a consideration of the Hodge types and the above part our claim follows immediately from the following lemma.
\end{proof}
\begin{lemma}
There is no Hodge class in $H^1(C)\otimes H^1(\widetilde E)$.
\end{lemma}
\begin{proof}
Assume $H^1(C)\otimes H^1(\widetilde E)$ has Hodge classes. Then there is a non-trivial morphism $J(C)\longrightarrow
J(\widetilde E)$
and thus an isogeny $J(\widetilde E)\longrightarrow J(C)\times E'$ for some elliptic curve $E'$. We denote also with
$\widetilde E$ the image of $\widetilde E$ under the composition
\[
\xymatrix{
\widetilde E \quad \ar@<-3pt>@{^{(}->}[r] & J(\widetilde E)=:\widetilde J \ar@{->>}[r] & J(C)\times E'.
}
\]
Exactly as before we have
\[
0\neq [\widetilde E] \in H^1(J(C))\otimes H^1(E') \simeq H^1(C)\otimes H^1(E').
\]
We will now prove that $H^1(E')=H^1(E)$ and this provides a contradiction with our condition about the Hodge classes of
$H^1(C)\otimes H^1(E)$  (section \ref{fa}).\\
Since $E$ is generic we have an isogeny $\widetilde J\longrightarrow P\times E$ (cf. Remark (4.7) pg. 228 in
\cite{CvGTiB}) for some simple abelian surface $P$. From this and since $\widetilde J$ is isogenous to $J(C)\times E'$
we have an isogeny $J(C)\times E'\longrightarrow P\times E$. From the non-triviality of
\[
\xymatrix{
E' \quad \ar@<-3pt>@{^{(}->}[r] & J(C)\times E'\ar[r] & P\times E \ar@{->>}[r] & E
}
\]
we can see that  $E'$ and $E$ are isogenous. We conclude $H^1(E)=H^1(E')$.
\end{proof}

\subsection*{The cohomology class of $a$ remains nonzero when restricted to open subsets of $C\times \widetilde E$}

In order to conclude this stament we apply the following basic fact:
\begin{satz}\label{ausFk}
Let $X$ and $S$  be smooth algebraic varieties, $f :X\longrightarrow S$ be a smooth and proper morphism and $a\in H^m(X, \mathbb Q)$. Assume that there exists a subvariety $T\subset S$ with the following property:\\
For all non-empty open subsets $U$ of $T$, $i^*(a)\neq 0$, where $i: f ^{-1}(U) \hookrightarrow X$ is the inclusion map.\\
Then the following holds for all non-empty open subsets $V$ of $S$:
\[
j^*(a) \neq 0, \text{  where  }  j: f ^{-1}(V) \hookrightarrow X \text{ is the inclusion map.}
\]
\end{satz}
\begin{proof}
\cite{Fk} pg. 111.
\end{proof}
Next we show that the hypothesis of the above theorem is satisfied for our class $a$. To begin with, let $U\subset C\times E$ be open. We can assume that $Z:=C\times E-U$ is a subvariety of codimension $1$ (i.e. a prime divisor) because such sets determine a basis of the topology.\\
Let $T$ denote the compact K\"ahler manifold $C\times E$. Let $\tau : \widetilde T\longrightarrow T$ be the blow-up of
$T$ along $Sing(Z)$. We can consider $U=T-Z\subset T-Sing(Z)$ as an open subset of $\widetilde T$. We have the following
Gysin sequence (cf. \cite{Ku} pg. 82):
\[
\begin{xy}
 \xymatrix{
H^{0}(\widetilde Z) \ar[r]^{\phi} & H^2(\widetilde T) \ar[r]^{ \widetilde i^*} & H^2(U) ,
}
\end{xy}
\]
where $\widetilde Z=\widetilde T-U$ and $\widetilde i:U\hookrightarrow \widetilde T$ is the inclusion map. We know
(\cite{PS} Lemma 1.16) that $\phi$ is of type $(1,1)$. From this fact and from our condition about the Hodge classes in
$H^1(C)\otimes H^1(E)$ (section (\ref{fa})) we get from the commutative diagram
\[
 \begin{xy}
  \xymatrix{ H^0(\widetilde Z) \ar[r]^{\phi} & H^2(\widetilde T) \ar[r]^{\widetilde i^*} \ar@{=}[d] & H^2(U) \ar@{=}[d] \\
H^1(C)\otimes H^1(E) \ar@{^{(}->}[r] & H^2(T) \ar[r]^{i^*} & H^2(U)
   }
 \end{xy}
\]

(here $i:U\hookrightarrow T$ is the inclusion map) that the restriction of $i^*$ to
\[
H^1(C)\otimes H^1(E) \longrightarrow H^2(U)
\]
is injective. For this reason the following map
\[
i^*:H^1(C)\otimes H^1(E)\otimes H^{4}(J\times P) \longrightarrow H^2(U)\otimes H^{4}(J\times P)
\]
is injective. In particular we have
\[
0\neq i^*(a) \in H^{6} \left ((C\times E)\times (J \times P) \vert _U\right ) .
\]
Thus $a$ satisfies the hypothesis of the Proposition \ref{ausFk}. We then get
\begin{theorem}\label{PPP}
The primitive part of the cohomology of $Y$ in $H^2(S,R^{4}f _*\mathbb Q)$ (i.e. in $H^2(S,\mathbb P_4)$) remains non-zero when restricted to all open subsets of $S$.
\end{theorem}
\begin{proof}
The above consideration shows that $[Y]\in H^2(S,R^4f_*\Q )$ satisfies the hypothesis of the Proposition \ref{ausFk}.
\end{proof}

\section{About the higher infinitesimal invariant}

First of all we fix some notation. For our map $f:X\longrightarrow S$ let $f'$ be the restriction of $f$ to $X_T:=f^{-1}(T)\subset X$. Here $T\subset S$ is just as in our construction in section (\ref{Deg}). We get a commutative diagram
\[
\begin{xy}
 \xymatrix{
 X ¸\ar[d] _{f} & X_T \ar@{_{(}->}[l] \ar[d] ^{f'} \\
 S & T \ar@{_{(}->}[l]
   }
\end{xy}
\]
Consider the following commutative diagram
\[
\begin{xy}
 \xymatrix{
& Y\times _S X \ar[dl]_{p_1} \ar[dr] ^{p_2} \ar[dd]^{\phi} & \\
Y\ar[rd] _{\pi} & & X \ar[ld]^{f} \\
& S &
}
\end{xy}
\]
where $\pi :Y\longrightarrow S$ is smooth and projective and $Y$ is a smooth, projective variety of relative dimension $d$.\\
Let $\gamma$ be a section of $R^u\phi _* \O ^t_{Y\times _S  X}$ and $\overline{\gamma}$ its image in $R^u\phi _* \O ^t_{Y\times _S  X/S}$ under the morphism $R^u\phi _* \O ^t_{Y\times _S  X}\longrightarrow R^u\phi _* \O ^t_{Y\times _S  X/S}$.
With these notations we can formulate the following lemma:
\begin{lemma}\label{gamma}
A section $\gamma$ of $R^u\phi _* \O ^t_{Y\times _S  X}$ induces a morphism of 
Leray spectral sequences (i.e. $d_1\circ \overline{\gamma} _*= \overline{\gamma} _* \circ d_1$ on the left side):
\[
\begin{xy}
\xymatrix{
\E^{p,q}_1(r,Y) \ar@2{->}[r] \ar[d] _{\overline{\gamma} _*} & R^{p+q}\pi _* \O ^r _Y \ar[d] ^{\gamma _*} \\
\E^{p,q+u-d}_1(r+t-d,X) \ar@2{->}[r] & R^{p+q+u-d} f_* \O ^{r+t-d}_X
}
\end{xy}
\]
where $\overline{\gamma} _* =1_{\O ^p_S} \otimes (p_2)_*(p_1^*(-)\cdot \overline{\gamma} )$ on the left side and $\gamma _* = (p_2)_*(p_1^*(-)\cdot \gamma )$ on the right side.
\end{lemma}
\begin{proof}
\cite{Ike} Lemma 2.1.
\end{proof}
{\bf Example:} A cycle $\Gamma \in Ch^r(Y\times _SX/S)$ induces through its cohomology class
\[
 \gamma =cl(\Gamma )\in H^0(S, R^r\phi _*\O ^r_{Y\times _SX} )
\]
a morphism of spectral sequences: 
\[
\overline{\gamma}_*:\E^{p,q}_1(\bullet , Y)\longrightarrow \E _1^{p,q+r-d}(\bullet +r-d, Y).
\]
We will use Lemma \ref{gamma} in this way (look at the proofs of Lemmas \ref{Fs1} and \ref{Fs2-1}).\\
Our polarization $L\in H^0(S, R^2f_*\C )$ from section \ref{P-a} induces a section of $R^1f_*\O ^1_{X/S}$ and this induces morphisms (assuming $r+q \leq 4$)
\[
 u: \E _1 ^{p,q} (r,X)=\O ^p _{S}\otimes R^{p+q}f_*\O ^{r-p} _{X/S}\longrightarrow \O ^p _{S}\otimes R^{p+q+1}f_*\O ^{r-p+1} _{X/S} =\E ^{p,q+1}_1 (r+1,X) .
\]
From the hard Lefschetz theorem we get isomorphisms
\[
 \begin{xy}
  \xymatrix {
\E _1^{p,q}(r,X) \ar[rr]^{u^{4-r-q}}_{\simeq \quad} & & \E_1^{p,4-r}(4-q,X) .
  }
 \end{xy}
\]
\begin{definition}\label{primit}
The primitive part $\E_1^{p,q}(r,X)_{prim}$ of $\E_1^{p,q}(r,X)$ ($q+r\leq 5$) is defined as follows:
\[
\E_1^{p,q}(r,X)_{prim}:= Ker \left \{ u^{5-r-q}: \E _1^{p,q}(r,X) \longrightarrow \E_1^{p,5-r}(5-q,X) \right \} .
\]
\end{definition}
{\bf Remark:} Since $u=1_{\O ^p_S}\otimes L$, where $L$ is as in section \ref{P-a}, we have
\begin{equation}
\E^{p,q}_1(r,X)_{prim}\simeq \O ^p_S\otimes  \left( R^{p+q}f_*\O ^{r-p}_{X/S} \right) _{prim}. \label{uprim}
\end{equation}

\subsection{The second infinitesimal invariant $\delta_2$ of $\alpha$}

From now on $\alpha$ represents the component of $Y\in Ch^3(X/S)$ in $Ch^3 _{(2)}(X/S)$ under the decomposition of Beauville (cf. \cite{Be2} or \cite{DeMu}). We want to show in this section that $\delta _2(\alpha) \neq 0$. 

\subsection*{The non-triviality of $\delta _2(\alpha)$}

We get from the above remarks on the Leray spectral sequence and our definition of $\alpha$ that $cl_{}(\alpha)$ is contained in $H^2(S,R^4f_*\C)$ and from (\ref{PPP}) we know that this class is non-trivial. This information is very useful for the proof of the following:
\begin{lemma}\label{infini}
$\delta _2 (\alpha )\neq 0$. Moreover
\[
0\neq \delta _2(\alpha )\vert _T \in H^0(T,\E ^{2,1}_1(3,X_{T})_{prim}).
\]
\end{lemma}
\begin{proof}
It is clear that the first claim follows from the second one, for this reason we prove only the second one.\\
The second claim needs some explanations because originally $\delta _2(\alpha)$ is an element of $H^0(S,
\E^{2,1}_2(3,X))$ and there is no canonical way to lift it to $H^0(S, \E^{2,1}_1(3,X))$. Restricted to $T$, the
Gau{\ss}-Manin connection ${\nabla}$ is zero because of the triviality of the family $X_T\longrightarrow T$. Therefore
and since $d_1=\overline{\nabla}$ we get
\[
\E ^{2,1}_2(3,X_T) \simeq \E ^{2,1}_1(3,X_T).
\]
We have the following commutative diagram: {\small 
\[
\begin{xy}
\xymatrix{
Ch^3_{(2)} (X/S) \ar[d]_{cl_{\C}} \ar@{^{(}->}[r]^{}  &F^2_SCh^3(X/S) \ar[r]^{\delta _2\quad} & H^0(S,Gr _L^2R^3f_*\O ^3_X) \ar@{=}[d] \\
H^2(S,R^4f_*\mathbb C) \ar[d]_{res \vert T}  & & H^0(S,\E ^{2,1}_2(3,X)) \ar[d]^{res \vert T} \\
H^2(T)\otimes H^4(X_0,\mathbb C) \ar[d] _{proj}&  & H^0(T,\E^{2,1}_2(3,X_{T})) \ar@2{-}[d]^{d_1=\overline {\nabla} = 0} \\
H^0(\O ^2_T)\otimes H^3(X_0, \O ^1_{X_0}) \ar@{=}[r]^{\simeq } & H^0(T, \O ^2_T\otimes R^3f'_*\O ^1_{X_{T}/T})\ar@{=}[r]^{} & H^0(T,\E^{2,1}_1(3,X_{T}))
}
\end{xy}
\] }
Let $cl_T(3,1)$ be the composition of the maps in the left column and $\delta _2 (\cdot )\vert _T$ the composition of the maps in the top row with those in the right column.\\
We know from (\ref{P-a}) that
\begin{align}
0\neq cl_T(3,1)(\alpha ) \in H^0(\O ^2_T) & \otimes H^3(X_0,\O ^1_{X_0})_{prim}\notag \\
 & =H^0\left ( T,\O ^2_T\otimes ( R^3f_*\O^1_{X _T/T})_{prim} \right )\notag \\ 
 &  = H^0(T,\E ^{2,1}_1(3,X_{T})_{prim}) .\notag
\end{align}
Using this we get
\begin{equation}
0 \neq \delta _2(\alpha )\vert _T \in H^0(T,\E ^{2,1}_1(3,X_{T})_{prim}). \label{D}
\end{equation}
\end{proof}
This property plays a very important role in the next section.

\section{About the higher Griff{}iths group}

This section is based on the ideas in \cite{Ike} but we give some simplifications of the proofs in our case.\\
Here we want to show that our cycle $\alpha$ induces a non-trivial element of
\[
Griff ^{3,2}(X_s)=\frac {F^2Ch^3(X_s)} { F^{3}Ch^r(X_s)+Z_0F^2Ch^3(X_s)} 
\]
where $s\in  S$ is the generic point. We divide this task into a couple of lemmas. We know already (Proposition \ref{mmm}) that $\alpha \in F^2_SCh^3(X/S)$ and from this that $\alpha _s \in F^2Ch^3(X_s)$.
\begin{lemma} \label{Fs1}
If $s\in S$ is the generic point then $\alpha _s \notin F^3Ch^3(X_s)$.
\end{lemma}
\begin{proof}
We assume $\alpha _s \in F^3Ch^3(X_s)$; this means (section \ref{Saitofil}) that
\[
\alpha =\Gamma _*\beta
\]
for some cycle $\beta \in F^2Ch^{3-q+d}(Y)$, where $d=\dim Y$, and  $Y$ is a projective and smooth variety  and $\Gamma \in Ch^q(Y\times  X_s)$ has the property
\[
\Gamma _*\left ( H^{2d-2q+4}(Y) \right ) \subset \left ( F^{2}H^{4}(X_s) \right ) .
\]
Because $s\in S$ is generic by the Proposition \ref{standardarguments} there is an open set $U\subset S$, an \'etale
morphism $S'\longrightarrow U\subset S$ with a point $s'\in S'$ over $s\in U\subset S$, a projective smooth morphism
$Y_{S'}\longrightarrow S'$ whose fiber over $s'\in S'$ is isomorphic to $Y$  and  $\beta _{S'} \in
Ch^{3-q+d}(Y_{S'}/S')$ and $\Gamma _{S'}\in Ch^q(Y_{S'}\times _{S'}X_{S'}/S')$ (here $X_{S'}:=X\times _SS'$) whose
restrictions to $Y$ are $\beta$ and $\Gamma$ respectively and the relation $\alpha _{S'}=(\Gamma _{S'})_*\beta _{S'}$ is
satisfied, where $\alpha _{S'}$ is the pull-back of $\alpha$ by $X_{S'}\longrightarrow S'$. By Theorem \ref{PPP} we
can assume $U=S$ and by abusing the notation we put $S'=S$.\\
From our assumption on $\Gamma$ we know that 
\[
\Gamma _* :R^{3-q+d}f_*(\Omega ^{1-q+d}_{Y_S/S}) \longrightarrow R^3f_*(\Omega ^1_{X_S/S})
\]
is trivial. Therefore the induced map of spectral sequences
\[
\begin{xy}
\xymatrix{
& \E ^{2,1-q+d}_1(3-q+d,Y_S) \ar[rr]^{\qquad (1_{\O^2_S})\otimes \Gamma _*} & &  \E ^{2,1}_1(3,X_S) & 
}
\end{xy}
\]
is trivial and the same holds for the induced map
\[
\begin{xy}
\xymatrix{
\E ^{2,1-q+d}_2(3-q+d,Y_S)  \ar[rrr] ^{\qquad \left [ (1_{\O^2_S})\otimes \Gamma _*\right ]} & &  &  \E _2^{2,1}(3,X_S). }
\end{xy}
\]
Therefore the induced morphism 
\[
\gamma _* : H^0(S, Gr^2_LR^{3-q+d}\pi _*\O ^{3-q+d}_{Y_S}) \longrightarrow H^0(S,Gr_L ^2R^3f_*\O ^3_{X_S})
\]
is trivial too (here we use $\E ^{2,\bullet} _2 =\E ^{2,\bullet}_{\infty}=Gr^2_L$).\\
From the following commutative diagram
\[
\begin{xy}
\xymatrix{
 F^2_SCh^{3-q+d}(Y_S/S) \ar@/_/[dd]_{\delta _2}  \ar[rr]^{\Gamma _*} &  &  F^2_SCh^3(X_S/S)  \ar@/^/[dd]^{\delta _2} \\
 & & \\
 H^0(S,Gr^2_LR^{3-q+d}\pi _* \O ^{3-q+d}_{Y_S}) \ar[rr]^{ \quad \gamma _*=0} &   & H^0(S, Gr_L^2R^3f_*\O ^3_{X_S}) 
}
\end{xy}
\]
and from Lemma \ref{infini} it follows that
 $ 0\neq \delta _2(\alpha _S)=\gamma _*\delta _2(\beta _S)= 0$ and this is a contradiction.
\end{proof}
We have assumed only that $\delta _2(\alpha) \neq 0$ but in fact $0\neq \delta (\alpha ) \vert _T\in H^0(T,
\E^{2,1}_1(3,X)_{prim})$. \\
We want to use this condition. First we remark the following:
\begin{lemma}\label{d1}
The map
\[
1_{\O^2_S}\otimes \overline{\nabla}: \E^{2,1}_1(3,X)_{prim} \longrightarrow \O ^2_S\otimes \left( \O ^1_S\otimes R^4f_*(\O ^0 _{X/S}) \right )
\]
is an isomorphism.
\end{lemma}
\begin{proof}
We will show that the morphism
\[
\overline{\nabla}:\left( R^3f_*\O ^1_{X/S} \right) _{prim}\longrightarrow \O ^1_S\otimes R^4f_*\O ^0_{X/S}
\]
is an isomorphism. We do this fiberwise. To begin with, let $\omega \in H^1(X_s,\O ^1_{X_s})$ be the class of the polarization of $X_s$ and $H^1(X_s,T_{X_s})_{\omega}=\left \{ \eta \in H^1(X_s,T_{X_s}): \eta \wedge \omega =0 \right \}$. We have the following exact sequence
\[
 \begin{xy}
  \xymatrix{
0\ar[r] & H^1(X_s,\O ^3_{X_s})_{prim} \ar[r] & H^1(X_s,\O ^3_{X_s}) \ar[r] ^{\wedge \omega} & H^2(X_s,\O ^4_{X_s}) \ar[r] & 0. }
 \end{xy}
\]
From this we get using Poincar\'e duality that
\[
 \begin{xy}
  \xymatrix{
0\ar[r] & H^2(X_s,\O ^0_{X_s}) \ar[r] ^{\wedge \omega }& H^3(X_s,\O ^1_{X_s}) \ar[r]  & H^1(X_s,\O ^3_{X_s})_{prim}^*
\ar[r] & 0. }
 \end{xy}
\]
This shows that
\begin{equation}
H^1(X_s,\O ^3_{X_s}) ^*_{prim}=H^3(X_s,\O ^1_{X_s})_{prim} . \label{Ceresa}
\end{equation}
From Lemma 2.2 of \cite{Ce} we know that the wedge product
\[
 H^1(X_s,T_{X_s})_{\omega}\otimes H^0(X_s,\O_{X_s}^4) \longrightarrow H^1(X_s,\O ^3_{X_s})_{prim}
\]
is surjective. By dualising we get, with help from (\ref{Ceresa}), an inclusion
\[
H^3(X_s,\O ^1_{X_s}) _{prim} \hookrightarrow H^1(X_s,T_{X_s})^*_{\omega}\otimes H^4(X_s,\O _{X_s}^0)\simeq \O ^1 _{S,s}\otimes H^4(X_s,\O _{X_s}^0) ,
\]
where the right isomorphism  is induced from the  Kodaira-Spencer map. By a theorem of Griff{}iths (cf. \cite{Vo} II pg. 136) that inclusion is exactly $\overline{\nabla}$. Now our claim follows from
\[
 \dim {H^3(X_s,\O ^1_{X_s})_{prim}} = 10 = \dim {\left( \O ^1_{S,s}\otimes H^4(X_s,\O ^0_{X_s})\right)}.
\]
\end{proof}
From the Lefschetz decomposition for $R^3f_*\O ^1 _{X/S}$ we can decompose $\E^{2,1}_1(3,X)$ as follows:
\[
\E^{2,1}_1(3,X)\simeq \E^{2,1}_1(3,X)_{prim} \oplus \mathcal V.
\]
With the above notation we have:
\begin{lemma}\label{Fs2-1}
If $\alpha \in Z_0F^2Ch^3(X/S)$ then $\delta _2(\alpha)\in H^0(S, \mathcal F)$ where
 \[
 \mathcal F:= \frac{\mathcal V \cap \ker \left( d_1: \E^{2,1}_1(3,X)\longrightarrow \E ^{3,1}_1(3,X) \right) }{\mathcal V \cap Im \left( d_1 : \E^{1,1}_1(3,X)\longrightarrow \E^{3,1}_1(3,X) \right ) }
\]
and $d_1$ is the differential of the spectral sequence.
\end{lemma}
\begin{proof}
We assume the contrary, namely that (remember the definition of $Z_0$ in section \ref{Saitofil})
\[
\alpha =\Gamma _*\beta ,
\]
where $\beta \in F^2Ch^d(Y/S)$, $Y$ is a smooth projective variety of relative dimension $d$, and $\Gamma \in Ch^3(Y\times _SX/S)$.\\
Since $d=\dim (Y/S)$ it follows that $R^{d+1}f_*(\Omega ^{d-3}_{Y/S})= 0$ and from this and the commutative diagram
\[
\begin{xy}
\xymatrix{
 \E^{2,d-2}_1(d,Y)=\O ^2_{S}\otimes R^df_*(\Omega ^{d-2}_{Y/S}) \ar@/_/[dd]_{1_{\O ^2_S}\otimes \overline{\nabla}} \ar[r]^{1 _{\O ^2_S}\otimes \Gamma _*}  &\O ^2_{S}\otimes R^3f_*(\O ^1 _{X/S})=\E^{2,1}_1(3,X) \ar@/^/[dd]^{1_{\O ^2_S}\otimes \overline{\nabla}}  \\
& & & \\
0=\O ^2_{S}\otimes \left( \O ^1_S\otimes R^{d+1}f_*(\O _{Y/S}^{d-3}) \right ) \ar[r]^{ 1 _{\O ^2_S}\otimes \Gamma _*}  & \O _{S}^2\otimes \left( \O ^1_S \otimes R^4f_*(\O ^0 _{X/S})\right )
}
\end{xy}
\]
we conclude that
\begin{equation}
 \left( 1_{\O ^2_S\otimes \overline{\nabla}}\right ) \left (  1_{\O ^2_S}\otimes \Gamma _*\right ) \left(\E^{2,d-2}_1(d,Y)\right )= 0. \label{o}
\end{equation}
Let
\begin{align}
\E _{prim}(\Gamma ) & := \left \{ \left (  1_{\O ^2_S}\otimes \Gamma _*\right ) \left(\E^{2,d-2}_1(d,Y)\right ) \right \} \cap \E^{2,1}_1(3,X)_{prim} .\notag
\end{align}
From Lemma \ref{d1} and $\E _{prim}(\Gamma)\subset \E ^{2,1}_1(3,X)_{prim}$ we conclude using (\ref{o}) that
\[
 \E _{prim}(\Gamma ) =0.
\]
Therefore the image of (we have taken cohomology)
\begin{equation}
\left [ \left( 1_{\O ^2_S} \right )\otimes \overline{\nabla} \right ] : \E^{2,d-2}_2 (d,Y)\longrightarrow \E ^{2,1}_2(3,X) \label{g-in-co}
\end{equation}
is contained in $\mathcal F$. Let $\gamma _*$ be the map that (\ref{g-in-co}) induces on global sections. Now our claim follows from the commutativity of the following diagram
\[
\begin{xy}
 \xymatrix{
 F^2Ch^d(Y/S) \ar[d] _{\delta _2} \ar[r]^{\Gamma _*} & F^2Ch^3(X/S) \ar[d] ^{\delta _2} \\
 H^0\left ( S,\E^{2,d-2}_2(d,Y)\right ) \ar[r]^{\gamma _*} & H^0\left ( S,\E^{2,1}_2(3,X) \right ) .
   }
\end{xy}
\]
\end{proof}
Since $d_1$ is trivial over $T$ it is clear from our definition of $\mathcal F$ in Lemma \ref{Fs2-1} that $\mathcal
F\vert _T=\mathcal V \vert _T$. This observation is a key point for the following
\begin{lemma}\label{Fs2}
$\alpha \notin Z_0F^2Ch^3(X/S)$.
\end{lemma}
\begin{proof}
Assume it is false. Then we get from Lemma \ref{Fs2-1} that
\[
 \delta _2(\alpha)\vert _T \in H^0(T, \mathcal F\vert _T)=H^0(T,\mathcal V \vert _T)
\]
but this (remember the definition of $\mathcal V$) contradicts Lemma \ref{infini}.
\end{proof}

\subsection{Main result}

We present now the main result of the paper.
\begin{theorem}\label{a4}
The element $\alpha \in Ch^3_{(2)}(X/S)$ gives a non-trivial element in
\[
Griff ^{3,2}(A^4) ,
\]
where $A^4$ denotes the generic abelian fourfold with polarization of type $(1,2,2,2)$.
\end{theorem}
\begin{proof}
Because the map $S\longrightarrow \mathcal A_4(1,2,2,2)$ is dominant (cf. \cite{BCV} Thm. 2.2) we see that $A^4$ can be realized as $X _s$ with $s\in S$ the generic point. So the element we are looking for is $\alpha _s$.\\
We see from Lemma \ref{Fs1} that we only need to check that $\alpha _s\notin Z_0F^2Ch^3(X_s)$.\\
Assume $\alpha _s\in Z_0F^2Ch^3(X_s)$. Again by Proposition \ref{Z0-F}, since $s\in S$ is generic, we get an open
set $U\subset S$ and an \'etale map $\pi :S'\longrightarrow U\subset S$ such that $\alpha _{S'}:=\pi ^*(\alpha )\in
Z_0F^2Ch^3(X_{S'}/S')$ where $X_{S'}=X\times _{S'}S$. Now we proceed exactly as in the proofs of the Lemmas
\ref{Fs2-1} and \ref{Fs2}.
\end{proof}
{\bf Remark:} The theorem holds for the generic point $s\in S$ but not over a Zariski open subset. However it will hold over the complement of a countable union of proper Zariski closed subsets. Sometimes such points are called very general. The reason is that the condition on the Hodge classes will fail over countably many proper closed subsets.


\end{document}